\documentclass[a4paper,12pt]{amsart}
\usepackage[utf8]{inputenc}
\usepackage{amsmath, amsthm, amscd, amsfonts, amssymb, graphicx, color}
\usepackage[bookmarksnumbered, colorlinks, plainpages]{hyperref}
\usepackage{wrapfig}
\usepackage{subfig}
\usepackage{subfig}
\usepackage{cite}
\usepackage[framemethod=default]{mdframed}
\usepackage{showexpl}
\usepackage[algoruled]{algorithm2e}

\newtheorem{rmk}{Remark}

\begin{document}
	
	 \newcommand{\be}{\begin{equation}}
	 \newcommand{\ee}{\end{equation}}
	 \newcommand{\bt}{\beta}
	 \newcommand{\al}{\alpha}
	 \newcommand{\laa}{\lambda_\alpha}
	 \newcommand{\lab}{\lambda_\beta}
	 \newcommand{\no}{|\Omega|}
	 \newcommand{\nd}{|D|}
	 \newcommand{\Om}{\Omega}
	 \newcommand{\h}{H^1_0(\Omega)}
	 \newcommand{\lt}{L^2(\Omega)}
	 \newcommand{\la}{\lambda}
	 \newcommand{\ro}{\varrho}
	 \newcommand{\cd}{\chi_{D}}
	 \newcommand{\cdc}{\chi_{D^c}}
	 \newtheorem{thm}{Theorem}[section]
	 \newtheorem{cor}[thm]{Corollary}
	 \newtheorem{lem}[thm]{Lemma}
	 \newtheorem{prop}[thm]{Proposition}
	 \theoremstyle{definition}
	 \newtheorem{defn}{Definition}[section]
	 \newtheorem{exam}{Example}[section]
	 \theoremstyle{remark}
	 \newtheorem{rem}{Remark}[section]
	 \numberwithin{equation}{section}
	 \renewcommand{\theequation}{\thesection.\arabic{equation}}
	 \numberwithin{equation}{section}
	\newtheorem{alg}{Algorithm}

\newcommand{\R}{\mathbb{R}}
\renewcommand{\div}{\operatorname{div}}
\newcommand{\st}{\;:\;}
\newcommand{\dx}{\,\mathrm{d}x}
\renewcommand{\d}{\mathrm{d}}
\newcommand{\norm}[1]{\left\|#1\right\|}
\newcommand{\mean}{\operatorname{mean}}

	\title[The second  eigenvalue   of  $p$-Laplace operator ]{Approximation of the second  eigenvalue   of  the  $p$-Laplace operator    in symmetric domains}
	
\author{Farid Bozorgnia}
\address{CAMGSD, Department of Mathematics, Instituto Superior T\'{e}cnico, Lisbon}
\email{bozorg@math.ist.utl.pt}

	 \date{\today}

\vskip12mm

\maketitle

\begin{abstract}
 A new idea to approximate the second eigenfunction and  the second eigenvalue of $p$-Laplace operator is given. In the case of  Dirichlet boundary condition, the scheme has the restriction      that the positive  and the negative part of the second eigenfunction have equal $L^p$-norm, however in the case of   Neumann boundary condition,  our algorithm has   not such  restriction.    Our algorithm generates a descending
sequence of positive numbers that  converges  to the second eigenvalue.  We give  various examples and computational tests.

 \end{abstract}
\[
\]



\bigskip

\section{Introduction}

 The $p$-Laplace operator
is a  homogeneous   nonlinear  operator  which arises    frequently in various applications in  physics, mechanics, and
image processing. Derivation of the $p$-Laplace operator from a nonlinear Darcy law and the continuity equation has been described in \cite{BGKT}.
 The eigenvalues  and the corresponding eigenfunctions of the $p$-Laplace operator   have  been much discussed   in the  literature,  due to mathematical challenges,   open questions    cf.  \cite{ABP, Lind,  Lindd, DKN},     and    regarding related  applications in image processing we refer to \cite{ETT, GG}. In particular, the second eigenvalue and eigenfunction  have been studied extensively, see   \cite{AT, Pa, PGP}.

 For dimension higher than one, except for  the case $p=2$, the
structure of higher  eigenfunctions are not well understood. For the    one-dimensional case  see \cite{FP, DM,  Lin}. In this work , we are interested in numerical approximation of the  second eigenvalue  $\lambda_2$ and the corresponding eigenfunction  of the  $p$-Laplace operator.    We treat the second eigenvalue as
an  optimal bi-partition of the first eigenvalue.
 There are various problems in mathematical physics and probability theory related
to  optimal partitioning of   the first  eigenvalue,  see  \cite{BH1,BB, BP}.  In \cite{HO}  investigated
Constrained Descent Method and the Constrained Mountain Pass Algorithm
  to approximate  the two smallest eigenvalue for $1.1 \le p \le 10.$

The structure of this paper is as follows. In Section 2, we review mathematical background and characterization of the second eigenvalue. Next in Section  3, we present our numerical approximation along the proof of convergence.  In Section 4,  as  an application of our algorithm,  we study  the spectral  clustering.  Last section  deals with the numerical implementation.

\section{Mathematical Background }

In this section,  we briefly  review some known results about the first and second eigenfunction
of $p$-Laplace operator with zero Dirichlet boundary condition in bounded domain.

For $ 1\leq p <\infty$ the first eigenvalue of  the $p$-Laplace operator in $W^{1,p}_{0}(\Omega),$ denoted by $\lambda_{1,p} (\Omega)$   is   given by
 \begin{equation}\label{n2}
\lambda_{1,p}(\Omega):=\underset{ u\neq 0}{\underset{u \in W^{1,p}_{0}(\Omega)}{ \min}} \frac{\int_{\Omega} |\nabla u(x)|^{p} dx }{\int_{\Omega}| u(x)|^{p}dx}.
 \end{equation}
 For every $1<p< \infty,$ the first eigenvalue is simple and isolated and the first eigenfunction doesn't change the sign.    The corresponding minimizer satisfies the Euler-Lagrange equation
\begin{equation}\label{2}
\left \{
\begin{array}{ll}
 - \Delta_{p} u =\lambda |u|^{p-2} u    & \text{in}    \quad  \Omega,\\
 u=0        & \text{on }    \quad \partial  \Omega.\\
  \end{array}
\right.
\end{equation}
Here $ \Delta_{p} u=  \text{div}(| \nabla u|^{p-2} \nabla u)$ which for $p=2,$ we have Laplace operator.

\begin{defn}\label{piegen}
A  non zero function  $u\in W^{1, p}_{0}(\Omega) \cap C(\overline{\Omega}),$ is called a $p$-eigenfunction in the weak sense  if there exist a $\lambda \in \mathbb{R}$
such that
\begin{equation}\label{p2}
\int_{\Omega}|\nabla u|^{p-2} \nabla u \cdot \nabla\phi \, dx=\lambda \int_{\Omega}| u |^{p-2}  u \, \phi \, dx,  \quad \forall \, \phi \in W^{1,
p}_{0}(\Omega).
\end{equation}
\end{defn}
The associated number $\lambda$ is called a $p$-eigenvalue.   It is not, however, known whether every such quantity is a "variational eigenvalue" like  for
the case $p = 2.$   In \cite{AT} Anane and Tsouli  gave a characterization of the variational eigenvalues of Problem (\ref{2}) by  the following minimax principle.  To do this,   first  define
 Krasnoselskii genus of a set  $A   \subseteq  W^{1,p}_{0}(\Omega) $     by
 \[
 \gamma(A)=\textrm{min}  {\{ k \in \mathbb{N} : \text{ there exist   }  f: A\rightarrow \mathbb{R}^{k}\setminus{0}, f \, \text{continuous and odd} }\}.
 \]
For $ k \in \mathbb{N}$ define
\[
\Gamma_{k} := {\{ A \subseteq  W^{1,p}_{0}(\Omega),\,   \text{symmetric, compact  and}\, \gamma(A)\ge k }\}.
\]
Then the eigenvalues of the $p$-Laplace  are
 \begin{equation}\label{hgh}
\lambda_{k,p}(\Omega)=\underset{ A\in \Gamma_{k}}{ \min  } \underset{u \in A}{ \sup} \frac{\int_{\Omega} |\nabla u(x)|^{p} dx }{\int_{\Omega}| u(x)|^{p}dx},
  \end{equation}
which   satisfying
 \[
 0 < \lambda_{1} < \lambda_{2} \le \cdots \le \lambda_{k} \rightarrow  \infty,
 \]
 as $k$ tends to infinity, see \cite{AT, LS, Lind, Lindd}.

It is shown  by Anane and Tsouli that  $\lambda_2$  defined by (\ref{hgh}) is essentially  the
second eigenvalue of the Dirichlet   $p$-Laplace, means  that the
eigenvalue problem (\ref{2}) has no other eigenvalue between  $\lambda_1$  and  $\lambda_2$.

 In the one dimensional case, $\Omega=( a,b) \subset \mathbb{R},$  it is known that all  eigenvalues  are  simple and  the eigenfunction corresponding to
  $\lambda_n $  has exactly  $n + 1$ zeros,  counting the ending  boundary points $a, b$.    The eigenvalues can be computed explicitly  by  variational formula and
the corresponding eigenfunctions are obtained in terms of the Gaussian hypergeometric function, see \cite{FP,Lin}).

 Here, we consider  another characterization of  the second eigenvalue which we use for our numerical simulation.
\begin{defn}
 Given a bounded open set
$\Omega  \subset  \mathbb{R}^d,$  a class    bi-partition of $\Omega$ (or decomposition)  is   a family of pairwise  disjoint, open and connected  subsets   ${\{\Omega_{1}, \Omega_{2} }\}$
 such that
 \[
 \Omega_{1},  \Omega_{2} \subseteq  \Omega, \, \Omega_{1} \cap \Omega_{2} =\emptyset,\, \Omega=  \Omega_{1} \cup \Omega_{2}.
 \]
 By $\mathfrak{D}_{2}$ we mean the set of all bi-partition of $\Omega$.
 \end{defn}


  For  any  arbitrary partition
 $\mathfrak{D}= (\Omega_{1},\Omega_{2}) \in \mathfrak{D}_{2},$  we define
\[
\Lambda_{2} (\mathfrak{D}) =\max \left( \lambda_{1}(\Omega_{1}), \lambda_{1}(\Omega_{2})  \right).
\]
 Also  let  $\mathfrak{L_{2}} (\Omega)$  denote   the infimum  of
   $\Lambda_{2}(\mathfrak{D}),$   over all the bi-partition i.e.,
\begin{equation}\label{f15}
 \mathfrak{L_{2}} (\Omega)=\underset{\mathfrak{D} \in \mathfrak{D_{2}}}{\inf}\Lambda_{2}(\mathfrak{D}).
 \end{equation}
An  optimal  bi-partition  is  a partition which realizes the infimum   in (\ref{f15}). For    $p=2$ the optimal partition of the  first eigenvalue   has been studied  extensively, see \cite{BB,BH}.

We know that the second eigenfunction changes its  sign on the domain, i.e., the second eigenfunction can be written as $ u=u_{+}-u_{-},$
   where
    $$  u_{+}=\max(u,0), \, u_{-}=\max(-u,0).$$ Obviously
    \[
    u_{+},  u_{-} \geq 0,   \quad      u_{+} \cdot u_{-}=0  \quad   \text{  in }\ \Omega.
    \]
Nodal domains of $u$  denoted by   $\Omega_{+}$ and   $\Omega_{-}$  are defined as  the support of positive and negative part of $u$
    \[
     \Omega_{+}={\{ x\in \Omega: u(x)>0}\}, \quad  \Omega_{-}={\{ x\in \Omega: u(x)<0}\}.
       \]

The following Lemma in  \cite{PGP}   shows  existence for minimal two partitions and   implies that
\[
 \lambda_{2}(\Omega )= \mathfrak{L_{2 }}(\Omega).
 \]
\begin{lem}\label{f5}  There exists  $u \in  W^{1,p}_{0} (\Omega) $  such that  $ ({\{ u_{+}  > 0}\} , {\{ u_{-}  > 0}\} )$   achieves infimum  in (\ref{f15}). Furthermore,

\[
\lambda_{1}({\{    u_{+}  > 0}\} ) = \lambda_{1}({\{  u_{-}  > 0}\}).
\]
\end{lem}
It is also known that any eigenfunction associated to an eigenvalue different from $\lambda_{1}$  changes sign.  The following   properties for second eigenvalue  hold:
\begin{itemize}
\item
If  $  \Omega_1 \subseteq    \Omega_2 \subseteq    \Omega,$ then
\[
 \lambda_2(p, \Omega_2) \le  \lambda_2(p, \Omega_1).
\]
\item Let $\Omega$ be a bounded domain in   $\mathbb{R}^d,$ then the eigenfunction associated to  $\lambda_2(p, \Omega)$ admits exactly two nodal domains.
\item The Courant theorem implies that in the linear case $p = 2,$  the number of nodal domains of an eigenfunction associated to $\lambda_2$ is exactly $2.$
\item  In \cite{ADS, ABP} is  shown    that the
second eigenfunctions are not radial in ball.

\end{itemize}

The limiting cases $p\rightarrow 1 $ and $p\rightarrow \infty$ are  more   complicated   and requires tools from non smooth critical point theory and the concept  viscosity solutions. Note for the limiting case $p$ tends to one, there are
 several ways  to define  the
second eigenfunction of the $1$-Laplace operator which  it does not
satisfy many of the properties of the second eigenfunction of the p-Laplace
operator in general\cite{Pa}.   As  $p$ tends to one,   in some cases the second eigenfunction
takes the form
\[
u_2=c_{1} \chi_{C_{1}} - c_{2}  \chi_{C_{2}}.
\]
Here $\chi_{A}$ is the characteristic function of given set $A$, the pair  $(C_1, C_2)$  is  so called
Cheeger-$2$-cluster of $\Omega$, while in other cases functions of that type can't be
eigenfunctions at all.

In the limiting case as $p$ tends to infinity,  the second eigenvalue has a geometric characterization.  Following    Section 4 of  \cite{JL}  define
\begin{equation}\label{98}
\Lambda_{2} \ = \ \frac{1}{r_{2}},
\end{equation}
where
$$r_{2}= \sup \{ r \in \mathbb{R}^+ : \text{there are  disjoint balls } B_1, B_2 \subset \Omega \text{ with radius } r\}.$$
Then  the following lemma from \cite{JL} states that the second eigenvalue of infinity Laplace is  $\Lambda_{2}$.
\begin{lem}
Let $\lambda_{2}(p)$ be the second $p$-eigenvalue in $\Omega.$ Then it holds that
\[
 \underset { p \rightarrow \infty }  {\lim}  \lambda_2(p)^{\frac{1}{p}}  \rightarrow  \Lambda_{2}.
\]
  $\Lambda_{2} \in \mathbb{R}$ is the second eigenvalue of the infinity Laplace.
\end{lem}

Furthermore,  it is  shown in \cite{JL} that the second eigenfunction of the infinity Laplace operator can be obtained as a viscosity solution of the following equation
\begin{equation}
\label{eq:homogeneous}
F_{\Lambda}(x, u,  \nabla u,  D^2 u) \ = \ 0 \quad \text{ for } x \in \Omega \ ,
\end{equation}
for which $F_\Lambda$ is given by
\begin{equation}
\label{eq:inf-lpc_eigenproblem_2}
 F_{\Lambda}(x, u , \nabla u , D^{2} u) \ = \
\left \{
\begin{array}{lrl}
\min{\{|\nabla u |- \Lambda u,  - \Delta_{\infty} u \}} &   u(x)>0,\\
-\Delta_{\infty} u &   u(x)=0,\\
\max{\{- |\nabla u |- \Lambda u,  - \Delta_{\infty} u \}} &   u(x) < 0.
\end{array}
\right.
\end{equation}
Here  $ \Lambda= \Lambda_2 \in \mathbb{R}$ denotes the second eigenvalue of the infinity Laplace  given by \eqref{98}.

\section{An iterative scheme }

In this section we discuss our algorithm to approximate the second eigenvalue $\lambda_2$  and corresponding  eigenfunction  denoted by $u$.  For Dirichlet boundary condition  our  main assumption is that domain $\Omega$ has the following  symmetric property
\[
\| u_{+}\|_{L^p(\Omega)}=\|  u_{-}\|_{L^p(\Omega)}.
\]

In \cite{BB2}  the authors studied gradient flows of $p$-homogeneous functionals on a Hilbert space and proved that after suitable rescaling the flow always converges to a nonlinear eigenfunction of the associated subdifferential operator. They also gave conditions for convergence to the first eigenfunction.

 The inverse power method is known to be an efficient method to approximate   the first eigenvalue of given operator see \cite{FB, LH,BMT}.   Our aim here is to extend this method combining with Lemma \ref{f5} to approximate the  second eigenvalue.

Following  Lemma  \ref{f5} and notation introduced before this Lema, we know that   restriction of  the second eigenfunction on each nodal domain is   the first eigenfunction,
i.e.,
\[
\lambda_{2} (\Omega)=\lambda_{1} (\Omega_{+})=\lambda_{1} (\Omega_{-}).
\]
So the second eigenvalue problem can be written as
\begin{equation}\label{228}
\left \{
\begin{array}{lll}
- \Delta_{p} (u_{+} -u_{-}) = \lambda_{1}(\Omega_{+})  u_{+}^{p-1}      -  \lambda_{1} (\Omega_{-})   u_{-}^{p-1}    &  \text{  in }\ \Omega, \\
 u_{+} = u_{-}=0         & \text{  on }    \partial  \, \Omega.
  \end{array}
\right.
\end{equation}
  The main steps are  as follows:

\begin{enumerate}
\item   Choose  arbitrary initial bi-partition  $\Omega_{+}^{0}$ and  $\Omega_{-}^{0}.$  Since  $\Omega_{+}^{0}$ and  $\Omega_{-}^{0}$ are disjoint and connected   components of $\Omega$ this prevent of obtaining higher eigenfunctions than second one.

 \item  Given    $\Omega_{+}^{k}$ and  $\Omega_{-}^{k},$   $k=0, 1, \cdots, $  obtain the  first eigenvalues  denoted by  $ \lambda_{1}^{k}(\Omega_{+}),$    $ \lambda_{1}^{k}(\Omega_{-})$  and
  first eigenfunctions  $ (u_{+}^{k},   u_{-}^{k} )$  normalized in $L^p$  related to $(\Omega_{+}^{k},  \Omega_{-}^{k}).$
     \item  Solve  the  following boundary value problem
     \begin{equation*}
\left \{
\begin{array}{lll}
- \Delta_{p} u = \lambda_{1}^{k}(\Omega_{+})  (u_{+}^{k})^{p-1}      -  \lambda_{1} (\Omega_{-})   (u_{-}^k)^{p-1}    &  \text{  in }\ \Omega, \\
u=0         & \text{  on }    \partial  \, \Omega.
  \end{array}
\right.
\end{equation*}

     \item  Update   $(\Omega_{+},  \Omega_{-})$ as supports of positive part and negative part of the solution $u$.
     \item Go to step (2).

     \end{enumerate}
     Note that in step 2, given  $\Omega_{+}^{k}$ and  $\Omega_{-}^{k},$  one needs to calculate $\lambda_{1}(\Omega_{+}^{k})$  and $ \lambda_{1}(\Omega_{-}^{k})$ and corresponding first eigenvalues for each sub-domain.  This step can be modified by  implementing  the  inverse power method along our Algorithm.  Assume    $u_{+}^{k}$ and $u_{-}^{k}\,$   are given   normalized in $L^p$    with   disjoint supports, then  define   values  $\lambda_{+}^{k}  $ and  $\lambda_{-}^{k}$  by
 \begin{equation}\label{sun5}
\left \{
\begin{array}{ll}
\lambda_{+}^{k}= \lambda_{+}^{k}(\Omega_{+}^{k})=  \int_{\Omega_{+}^{k}} |\nabla u_{+} ^{k}(x)|^{p}\, dx, &\\\\
   \lambda_{-}^{k}= \lambda_{-}^{k}(\Omega_{-}^{k})=  \int_{\Omega_{-}^{k}}
 |\nabla u_{-}^{k}(x)|^{p}\, dx. &
  \end{array}
\right.
\end{equation}
The   algorithm  to approximate  the second eigenvalue and the second eigenfunction  are    as follows:

{\scriptsize{
\begin{algorithm}[H]
	\SetAlgoLined

	\SetKwInOut{Input}{inputs}
	\SetKwInOut{Output}{output}
	\SetKwInOut{Required}{requied}
	
	\Indm
	\Input{$u^{0}=u_{+}^{0}- u_{-}^{0}, \epsilon$ .}
 \Output{Approximation of second eigenvalue and second eigenfunction.}
	\Indp
	\BlankLine

\[
\]

 \begin{enumerate}
 \item
    Set $ k=0$, choose  initial arbitrary  guesses
$u_{+}^{0}>0, u_{-}^{0}> 0$ having disjoint supports, normalized  in  $L^{p}(\Omega)$
  and vanishing on the boundary with $ \Omega_{+}$ and
$\Omega_{-}$ as the supports of  functions $u_{+}^{0}$ and $u_{-}^{0}$ respectively.

	\item 	Given   $ u^{k}= u_{+}^{k}-u_{-}^{k}$ where   $u_{+}^{k}$  and $u_{-}^{k}$ are
 normalized in $L^p,$   with  disjoint supports, then obtain $ \lambda_{+}^{k}$ and
  $\lambda_{-}^{k}$  by (\ref{sun5}). \;

\item Solve

\begin{equation}\label{sunshine86}
\left \{
\begin{array}{lll}
- \Delta_{p} u =   |u^{k}|^{p-2}   \bigg(\lambda_{+}^{k}  u_{+}^{k}   -\lambda_{-}^{k} u_{-}^{k} \bigg)   &  \text{in }\ \Omega, \\
 u =0        & \text{on }    \partial  \, \Omega.\\
  \end{array}
\right.
\end{equation}

\item Set  $u_{+}^{k+1}$  and  $u_{-}^{k+1}$ as  positive  and negative part  of   the solution of (\ref{sunshine86}).  Normalized  $u_{+}^{k+1}$  and  $u_{-}^{k+1}$ in $L^p$. Calculate $\lambda_{1}^{k+1}(\Omega_{1}),$   $\lambda_{1}^{k+1}(\Omega_{2}).$ \\

\item If $|\lambda_{1}^{k+1}(\Omega_{+})-\lambda_{1}^{k}(\Omega_{+})| \ge \epsilon,\,  \& \, |\lambda_{1}^{k+1}(\Omega_{-})-\lambda_{1}^{k}(\Omega_{-})| \ge \epsilon$
then

\item Set $k=k+1$ and  go step (2)

 \end{enumerate}
	\caption{Second eigenvalue algorithm }
	\label{Algorithm 1 )}
\end{algorithm}
}}

\begin{rmk}
  Note that $ \Omega_{+}$ and $\Omega_{-}$   change  in iterations but for simplicity  we write $ \Omega_{+}$ and $\Omega_{-}$  instead of $ \Omega_{+}^{k}$ and $\Omega_{-}^{k}$.
  \end{rmk}
  \begin{rmk}
Consider the  case $p=2$. Let denote the first eigenfunction by $w_1$.   The second
eigenvalue is given by
\begin{equation}\label{hg}
\lambda_{2}(\Omega)= \underset{u \perp w_1}{\text{inf}} \frac{\int_{\Omega} |\nabla u(x)|^{2} dx }{\int_{\Omega}| u(x)|^{2}dx}.
  \end{equation}

Assume the initial guess $u_0$ be chosen such that $ ( \lambda_{+}^{0}  u_{0}^{+} -  \lambda_{-}^{0}  u_{0}^{-} )  \perp w_1$ then the Algorithm generates sequence ${\{u_{n}}\}$ which are orthogonal to $w_1$. Multiply the equation
\[
- \Delta u =     \lambda_{+}^{0}(\Omega_{+}) u_{+}^{0}   -\lambda_{-}^{0}(\Omega_{-}) u_{-}^{0}
\]
by $w_1$ and integrating by parts  two times on left side, implies
\[
\int_{\Omega}   u(  \Delta w_1) \, dx=   \int_{\Omega} (  \lambda_{+}^{0}   u_{+}^{0}   -\lambda_{-}^{0} u_{-}^{0})  w_1 \, dx=0.\]
This shows
\[
\int_{\Omega}   u  \,  w_1 \, dx=0.
\]
\end{rmk}

\subsection{Neumann case}

In this part,  we consider  the $p$-Laplace  eigenvalue problem  with  Neumann  boundary.
\begin{equation}\label{evproblem}
\begin{cases}
-\Delta_p u=\lambda|u|^{p-2}u  \quad&  \text{in }\Omega,\\
|\nabla u|^{p-2}\nabla u\cdot\nu=0 \quad&   \text{on }\partial\Omega,
\end{cases}
\end{equation}
where  $ 1< p <\infty$  and    $\nu$ is the unit normal vector to $\partial\Omega$.

A  non zero function  $u\in W^{1, p} (\Omega)$ is called a $p$-eigenfunction   of (\ref{evproblem}) in the weak sense, if there exist a $\lambda \geq 0$
such that
\begin{equation}\label{weak_evproblem}
\int_{\Omega}|\nabla u|^{p-2} \nabla u \cdot \nabla\phi\, dx=\lambda \int_{\Omega}| u |^{p-2}  u \, \phi \, dx,  \quad   \, \phi \in W^{1,p}(\Omega).
\end{equation}
The number $\lambda$ is called a $p$-eigenvalue.  The  first eigenfunction is a constant function and $\lambda_1=0$.
Note that by testing \eqref{weak_evproblem} with $\phi=1$ we obtain that any eigenfunction with $\lambda>0$ necessarily fulfills
\begin{equation}\label{eq:zero_p-mean}
    \int_\Omega |u|^{p-2}u\, dx = 0.
\end{equation}
Equivalently,
\[
\| u_{+} \|_{L^{p-1}(\Omega)}=\| u_{-} \|_{L^{p-1}(\Omega)}.
\]
This motivates the following definition:

The {variational second $p$-eigenvalue} is defined as

\begin{equation}\label{second_eigenvalue}
\lambda_{2,p}(\Omega):=\inf\left\lbrace\frac{\int_{\Omega} |\nabla u(x)|^{p}\, dx }{\int_{\Omega}| u(x)|^{p}\, dx}\st \,\int_\Omega |u|^{p-2}u=0\right\rbrace.
 \end{equation}
Any $u\in W^{1,p}(\Omega)\setminus\{0\}$ realizing the infimum is  called {variational second $p$-eigenfunction}.

\begin{defn}
The $p$-mean of a function $u\in L^p(\Omega)$ is defined as
\begin{equation}\label{eq:pmean}
    \mean_p(u):=\frac{1}{|\Omega|}\int_\Omega|u|^{p-2}u\dx.
\end{equation}
\end{defn}

\begin{rem}
     The following  $p$-Laplace  problem
    \begin{equation}\label{eq:pPoisson}
        \begin{cases}
        -\Delta_p u=f \quad&\text{on }\Omega,\\
        |\nabla u|^{p-2}\nabla u\cdot\nu=0  \quad&\text{on }\partial\Omega,
        \end{cases}
    \end{equation}
    has a one parameter family of solutions; adding any  constant to a solution will be a solution.  Furthermore, the datum $f$ has to admit the \emph{compatibility condition} $\mean_2(f) = 0$.
\end{rem}

 The algorithm for Neumann case is  as following.

 {\scriptsize{
\begin{algorithm}[H]
	\SetAlgoLined

	\SetKwInOut{Input}{inputs}
	\SetKwInOut{Output}{output}
	\SetKwInOut{Required}{requied}
	
	\Indm
	\Input{$u^{0}=u_{+}^{0}- u_{-}^{0}, \epsilon$.}
 \Output{Approximation of second eigenvalue and second eigenfunction.}
	\Indp
	\BlankLine

	\While{$|\lambda_{1}^{k+1}(\Omega_{+})-\lambda_{1}^{k}(\Omega_{+})| \ge \epsilon,\,  \& \, |\lambda_{1}^{k+1}(\Omega_{-})-\lambda_{1}^{k}(\Omega_{-})| \ge \epsilon$ }{
\[
\]

 \begin{enumerate}
 \item
    Set $ k=0$.  Initialize with   arbitrary guess $u^{(0)}\in W^{1,p}(\Omega)$ with $\mean_p(u^{(0)})=0$.

	\item  For $k\geq 0$, we set $u_\pm^{k}:=\max(\pm u^{k)},0)$ and define
\begin{align*}
    \lambda_\pm^{k}&:=\frac{\int_\Omega|\nabla u_\pm^{k}|^p\dx}{\int_\Omega|u_\pm^{k}|^p\dx},\\
        f^{k}&:= \lambda_+^{k} (u_+^{k})^{p-1} -\lambda_-^{k}(u_-^{k})^{p-1}.
\end{align*}

\item \item Next define $u^{(k+1)}$ as the unique   solution to the problem

  \begin{equation}\label{p13}
  \left \{
\begin{array}{lll}
-\Delta_p u = f^{k}      &  \text{in }\Omega\\
|\nabla u|^{p-2}\nabla u\cdot\nu=0 &\text{on }\partial\Omega\\
\mean_p(u)=0.
  \end{array}
\right.
\end{equation}

\item  Set  $k=k+1$  and go back to (2).

 \end{enumerate}
	}
	\caption{Second eigenvalue for Neumann boundary }
	\label{Algorithm 2 )}
\end{algorithm}
}}

\begin{rem}
To  enforce the condition $\mean_p(u)=0$ numerically in  (\ref{p13}), we  fix  the value of the solution at a single node of the grid to an arbitrary value, which yields    a unique solution $\tilde{u}$. Then we set $u:=\tilde{u}-c$ where $c$ is such that $\mean_p(u)=0$ holds.
\end{rem}

\subsection{Convergence of the Algorithm}

  In the sequel,  for any $v\in W_{0}^{1,p}(\Omega)$  by $\tilde{v}$  we mean the normalized in $L^p(\Omega)$
$$
\tilde{v}=\frac{v}{\|v\|_{L^p(\Omega)}}.
$$
%
Furthermore, we define  $u^{k}_{+}, u^{k}_{-} \in W^{1,p}_0(\Omega)$ as  the  positive and negative parts of the solution of the following Dirichlet problem  inductively,
\begin{equation}\label{fa}
\left \{
\begin{array}{ll}
-\Delta_{p}  u^{k}=   \lambda_{+}^{k-1} \, (\tilde{u}_{+}^{k-1})^{p-1} - \lambda_{-}^{k-1}   \, (\tilde{u}_{-}^{k-1})^{p-1}   &    \textrm{ in } \Omega,\\
 u^{k} =0      &    \textrm{ on }      \partial  \, \Omega.
  \end{array}
\right.
\end{equation}
Here $\tilde{u}^{k-1}=\tilde{u}_{+}^{k-1}- \tilde{u}_{-}^{k-1}.$

\begin{prop}\label{sun200} With the  introduced  notations,    the following facts hold:
 \begin{align*}
  & |\tilde{u}^{k-1}|= |\tilde{u}_{+}^{k-1}- \tilde{u}_{-}^{k-1} |=  \tilde{u}_{+}^{k-1}+  \tilde{u}_{-}^{k-1},\\
  &   |\tilde{u}^{k-1}|^{p-2} =  (\tilde{u}_{+}^{k-1})^{p-2} + ( \tilde{u}_{-}^{k-1})^{p-2},\\
 &  \|\lambda_{+}^{k-1}     \tilde{u}_{+}^{k-1} -\lambda_{-}^{k-1}     \tilde{u}_{-}^{k-1}\|^{p}_{L^p} =(\lambda_{+}^{k-1} )^{p} + (\lambda_{-}^{k-1} )^{p},\\
 & \| \lambda_{+}^{k-1}   \nabla  \tilde{u}_{+}^{k-1} -\lambda_{-}^{k-1}    \nabla \tilde{u}_{-}^{k-1}\|^{p}_{L^p}=(\lambda_{+}^{k-1} )^{p+1} + (\lambda_{-}^{k-1} )^{p+1}.
\end{align*}
\end{prop}

The next Lemma shows the monotonicity of sequence of approximations of second eigenvalues as $p$ tends to one.

\begin{lem}
Let $\lambda_{+}^{k}(\Omega_+)$ and $  \lambda_{-}^{k}(\Omega_-)$ be  obtained by Algorithm 1. Then as $p  \rightarrow 1$ the following holds

\[
\max \left( \lambda_{+}^{k}(\Omega_{+}),  \lambda_{-}^{k}(\Omega_{-})\right)       \le \max \left( \lambda_{+}^{k-1}(\Omega_{+}),  \lambda_{-}^{k-1}(\Omega_{-})\right),
\]
for every $k\ge 1$.
\end{lem}
\begin{proof}

To start,   multiply  the first equation (\ref{fa}) by $ u^{k}_{+}  $ and integrate over $\Omega$ to deduce
\begin{equation} \label{eq1}
\begin{split}
\int_{\Omega} |\nabla  u^{k}_{+}  |^{p}  \, dx   & = \int_{\Omega}   u^{k}_{+}  \, |\tilde{u}^{k-1}|^{p-2}
    \bigg[ \lambda_{+}^{k-1}    \tilde{u}_{+}^{k-1} -\lambda_{-}^{k-1}     \tilde{u}_{-}^{k-1} \bigg]\, dx  \\
 & =  \int_{\Omega}   u^{k}_{+}  \,   \bigg[ \lambda_{+}^{k-1}   ( \tilde{u}_{+}^{k-1} )^{p-1}  -\lambda_{-}^{k-1}  (\tilde{u}_{-}^{k-1})^{p-1}\, \bigg]\, dx.
\end{split}
\end{equation}
H\"{o}lder inequality  on right hand side gives
\[
\int_{\Omega} |\nabla  u^{k}_{+}  |^{p}  \, dx    \le \| u^{k}_{+}\|_{L^p} \|\lambda_{+}^{k-1}
   ( \tilde{u}_{+}^{k-1} )^{p-1}  -\lambda_{-}^{k-1}    ( \tilde{u}_{-}^{k-1})^{p-1}\|_{L^{\frac{p}{p-1}}}.
\]
By Proposition (\ref{sun200}) we obtain
 \begin{equation} \label{sun22}
\|\nabla  u^{k}_{+} \|^{p}_{L^{p}} \le  \| u^{k}_{+}\|_{L^p} \bigg[(\lambda_{+}^{k-1} )^{\frac{p}{p-1}} + (\lambda_{-}^{k-1} )^{\frac{p}{p-1}} \bigg]^{\frac{p-1}{p}}.
\end{equation}
The same argument as above indicates
 \begin{equation} \label{sun23}
\|\nabla  u^{k}_{-}  \|^{p}_{L^{p}} \le  \| u^{k}_{-} \|_{L^p} \bigg[(\lambda_{+}^{k-1} )^{\frac{p}{p-1}} + (\lambda_{-}^{k-1} )^{\frac{p}{p-1}} \bigg]^{\frac{p-1}{p}}.
\end{equation}
The inequalities (\ref{sun22})  and (\ref{sun23}) imply
\[
\max \left(\frac{\|\nabla  u^{k}_{+}  \|^{p}_{L^{p}} }{\| u^{k}_{+}\|_{L^p}} , \frac{\|\nabla  u^{k}_{-}  \|^{p}_{L^{p}} }{\| u^{k}_{-}\|_{L^p}}\right)    \le   \bigg[(\lambda_{+}^{k-1} )^{\frac{p}{p-1}} + (\lambda_{-}^{k-1} )^{\frac{p}{p-1}} \bigg]^{\frac{p-1}{p}}.
\]
Let $p \rightarrow 1^{+}$ and using the fact
\[
 \underset{\alpha \rightarrow \infty}{\lim} ( a^{\alpha} + b^{\alpha})^{\frac{1}{\alpha}}=\max(a, b),
\]
complete the proof.
\end{proof}

\begin{lem}
With same assumptions as before the sequence $u^{k}$  is  bounded from below for every $p\ge 1.$
 \end{lem}
\begin{proof}
Multiply   equation (\ref{fa}) by $ u^{k},$  integrate over $\Omega$  and  H\"{o}lder inequality  give
\begin{equation}\label{f21}
\int_{\Omega} |\nabla  u^{k}   |^{p}  \, dx    \le \| u^{k} \|_{L^p} \|\lambda_{+}^{k-1}
   ( \tilde{u}_{+}^{k-1} )^{p-1}  -\lambda_{-}^{k-1}    ( \tilde{u}_{-}^{k-1})^{p-1}\|_{L^{\frac{p}{p-1}}}.
\end{equation}

\begin{equation*}
 \|\nabla  u^{k} \|^{p}_{L^{p}}   \le \| u^{k}\|_{L^p} \bigg[(\lambda_{1}^{k-1} )^{\frac{p}{p-1}} + (\lambda_{-}^{k-1} )^{\frac{p}{p-1}} \bigg]^{\frac{p-1}{p}}.
\end{equation*}

Note that $u^{k}$ satisfies
\[
\lambda_{2} \| u^{k}\|^{p}_{L^p} \le \|\nabla  u^{k} \|^{p}_{L^{p}}.
\]
This shows
\[
\lambda_{2} \| u^{k}\|^{p}_{L^p}
\le \| u^{k}\|_{L^p} \bigg[(\lambda_{1}^{k-1} )^{\frac{p}{p-1}} + (\lambda_{-}^{k-1} )^{\frac{p}{p-1}} \bigg]^{\frac{p-1}{p}},
\]
which implies
  \begin{equation} \label{sun32}
 \| u^{k}\|_{L^p} \le \frac{1}{\lambda_{2}}  \bigg[(\lambda_{1}^{k-1} )^{\frac{p}{p-1}} + (\lambda_{-}^{k-1} )^{\frac{p}{p-1}}  \bigg]^{\frac{1}{p}}.
 \end{equation}
Multiply (\ref{fa}) by $ \tilde{u}_{+}^{k-1}$ and integrate over $\Omega$
to deduce
$$
\int_{\Omega}   |\nabla u^{k}|^{p-2} \nabla u^{k} \cdot \nabla \tilde{u}_{+}^{k-1} \, dx    =
\lambda_{+}^{k-1}.
$$
Here we used the facts that
\[
 \tilde{u}_{+}^{k-1} \cdot\tilde{u}_{-}^{k-1}   =0,  \quad \quad \int_{\Omega} ( \tilde{u}_{+}^{k-1})^p  \, dx=1.
 \]
Form here we get
\[
\lambda_{+}^{k-1} \le \|\nabla u^{k}\|_{L^p}^{p-1}     \|\nabla \tilde{u}_{+}^{k-1}\|_{L^p},
\]
the same argument shows
\[
\lambda_{-}^{k-1} \le \|\nabla u^{k}\|_{L^p}^{p-1}     \|\nabla \tilde{u}_{-}^{k-1}\|_{L^p},
\]
Considering   $\|\nabla \tilde{u}_{+}^{k-1}\|_{L^p}= (\lambda_{+}^{k-1})^{\frac{1}{p}}$  and   $\|\nabla \tilde{u}_{-}^{k-1}\|_{L^p}= (\lambda_{-}^{k-1})^{\frac{1}{p}}$  implies

\begin{equation}\label{sun24}
\left \{
\begin{array}{ll}
 (\lambda_{+}^{k-1})^{\frac{p-1}{p}} \le  \|\nabla u^{k}\|_{L^p}^{p-1},   &  \\\\
 (\lambda_{-}^{k-1})^{\frac{p-1}{p}} \le  \|\nabla u^{k}\|_{L^p}^{p-1}.   &
  \end{array}
\right.
\end{equation}
 Inserting the inequalities (\ref{sun24}) in  (\ref{f21}) yields
 \begin{equation}\label{sun42}
 1 \le   2^{\frac{p-1}{p}} \, \| u^{k}\|_{L^p}.
\end{equation}
Also from inequalities (\ref{sun24}) and   (\ref{sun22}) it follows

\begin{equation}\label{sun27}
\left \{
\begin{array}{ll}
\|\nabla u^{k}_{+}\|_{L^p}^{p} \le 2^{\frac{p-1}{p}} \, \| u^{k}_{+}\|_{L^p} \|\nabla u^{k} \|_{L^p}^{p},\\\\
\|\nabla u^{k}_{-}\|_{L^p}^{p} \le 2^{\frac{p-1}{p}} \,  \| u^{k}_{-}\|_{L^p} \|\nabla u^{k}\|_{L^p}^{p}.
 \end{array}
\right.
\end{equation}
\end{proof}

\begin{lem}
  For every $p>1$ the following holds
\[
\lambda^{k}      \le \max \left( \lambda_{+}^{k-1}(\Omega_{+}),  \lambda_{-}^{k-1}(\Omega_{-})\right),
\]
where
\[
\lambda^{k} = \frac{\int_{\Omega}  | \nabla u^{k}|^{p} \, dx   }{\int_{\Omega}  | u^{k}|^{p} \, dx}.
\]
\end{lem}
\begin{proof}
We rewrite the right hand side of (\ref{f21}) as
\begin{equation}\label{fsun1}
\|\nabla  u^{k}\|^{p}_{L^{p}} \le
\| u^{k}\|_{L^p} \frac{\bigg\|\lambda_{+}^{k-1}   ( \tilde{u}_{+}^{k-1})^{p-1}  -\lambda_{-}^{k-1}    (\tilde{u}_{-}^{k-1})^{p-1}\bigg\|^{\frac{p}{p-1}}_{L^{\frac{p}{p-1}}}}
{ \bigg\|\lambda_{+}^{k-1}   ( \tilde{u}_{+}^{k-1})^{p-1}  -\lambda_{-}^{k-1}     (\tilde{u}_{-}^{k-1})^{p-1}\bigg\|^{\frac{1}{p-1}}_{L^{\frac{p}{p-1}}}   }
\end{equation}
Next by Proposition \ref{sun200} we have
\[
 \bigg\|\lambda_{+}^{k-1}   ( \tilde{u}_{+}^{k-1})^{p-1}  -\lambda_{-}^{k-1}     (\tilde{u}_{-}^{k-1})^{p-1}\bigg\|^{\frac{1}{p-1}}_{L^{\frac{p}{p-1}}}=  \bigg(  ( \lambda_{+}^{k-1}  )^{\frac{p}{p-1}}  +   ( \lambda_{+}^{k-1}  )^{\frac{p}{p-1}}\bigg)^{\frac{1}{p}}.
 \]
 Also
 \begin{align*}
  & \|\lambda_{+}^{k-1}  ( \tilde{u}_{+}^{k-1})^{p-1}  -\lambda_{-}^{k-1}     (\tilde{u}_{-}^{k-1})^{p-1}\|^{\frac{p}{p-1}}_{L^{\frac{p}{p-1}}}
= \int_{\Omega} (\lambda_{+}^{k-1} )^{\frac{p}{p-1}}    (\tilde{u}_{+}^{k-1})^{p}  +     ( \lambda_{-}^{k-1}  )^{\frac{p}{p-1}}    ( \tilde{u}_{-}^{k-1})^{p}\, dx\\
&  =  \int_{\Omega}   \bigg[\lambda_{+}^{k-1}  \,  (\tilde{u}_{+}^{k-1})^{p-1}- \lambda_{-}^{k-1}  \, (\tilde{u}_{-}^{k-1})^{p-1}  \bigg ]\, \bigg[ (\lambda_{+}^{k-1} )^{\frac{1}{p-1}} \,  \tilde{u}_{+}^{k-1}  - ( \lambda_{-}^{k-1}  )^{\frac{1}{p-1}} \, \tilde{u}_{-}^{k-1}\bigg]\,  dx\\
&  = \int_{\Omega}  (- \Delta_{p}  u^{k}  )\, [ (\lambda_{+}^{k-1} )^{\frac{1}{p-1}} \,  \tilde{u}_{+}^{k-1}
      -( \lambda_{-}^{k-1}  )^{\frac{1}{p-1}} \, \tilde{u}_{-}^{k-1}]\,
        \,   dx\\
&       =  \int_{\Omega}  | \nabla u^{k}|^{p-2}    \nabla u^{k} \cdot [(\lambda_{+}^{k-1} )^{\frac{1}{p-1}}  \nabla  \tilde{u}_{+}^{k-1} -  (\lambda_{-}^{k-1} )^{\frac{1}{p-1}} \lambda_{+}^{k-1}    \nabla \tilde{u}_{-}^{k-1}    ]\, dx\\
     &    \le \|\nabla u^{k}\|^{p-1}_{L^p} \| (\lambda_{+}^{k-1} )^{\frac{1}{p-1}}  \nabla  \tilde{u}_{+}^{k-1} - (\lambda_{-}^{k-1} )^{\frac{1}{p-1}}  \nabla \tilde{u}_{-}^{k-1}\|_{L^p}.
 \end{align*}
Inserting the last inequality  in  (\ref{fsun1}) and dividing by $\| u^{k}\|^{p}_{L^p}$  we obtain
\begin{equation*}
\frac{\int_{\Omega}  | \nabla u^{k}|^{p} \, dx   }{\int_{\Omega}  | u^{k}|^{p} \, dx  } \le
 (\frac{\int_{\Omega}  | \nabla u^{k}|^{p} \, dx   }{\int_{\Omega}   |u^{k}|^{p} \, dx  })^{\frac{p-1}{p}}\,
   \frac{\| (\lambda_{+}^{k-1})^{\frac{1}{p-1}}\,  \nabla  \tilde{u}_{+}^{k-1} -(\lambda_{-}^{k-1})^{\frac{1}{p-1}}\, \nabla \tilde{u}_{-}^{k-1}\|_{L^p}}
{(  ( \lambda_{+}^{k-1}  )^{\frac{p}{p-1}}  +   ( \lambda_{+}^{k-1}  )^{\frac{p}{p-1}} )^{\frac{1}{p}}}.
 \end{equation*}
The inequality above yields
\[
 \frac{\int_{\Omega}  | \nabla u^{k}|^{p} \, dx   }{\int_{\Omega}  | u^{k}|^{p} \, dx  } \le
\frac{\|(\lambda_{+}^{k-1})^{\frac{1}{p-1}}\,  \nabla  \tilde{u}_{+}^{k-1} -(\lambda_{-}^{k-1})^{\frac{1}{p-1}}\, \nabla \tilde{u}_{-}^{k-1}  \|^{p}_{L^p}}
{( \lambda_{+}^{k-1}  )^{\frac{p}{p-1}}  +   ( \lambda_{+}^{k-1}  )^{\frac{p}{p-1}}}.
\]
Thus
\begin{equation}\label{fsun10}
 \frac{\int_{\Omega}  | \nabla u^{k}|^{p} \, dx   }{\int_{\Omega}  | u^{k}|^{p} \, dx  } \le
  \frac{(\lambda_{+}^{k-1})^{\frac{p}{p-1}+1} + (\lambda_{-}^{k-1})^{\frac{p}{p-1}+1}}
  {( \lambda_{+}^{k-1}  )^{\frac{p}{p-1}}  +   ( \lambda_{+}^{k-1} )^{\frac{p}{p-1}}  }.
  \end{equation}

Form  \eqref{fsun10}  we infer
\begin{equation}\label{fa4}
 \lambda^{k}  \leq \max\left( \lambda_{+}^{k-1} ,  \lambda_{-}^{k-1} \right).
\end{equation}
\end{proof}

\section{Graph   $p$-Laplacian}
The aim of this part is to perform  data clustering  by using our algorithm. Given some data and a notion of similarity,  we aim to   partition  the input data into maximally homogeneous groups (i.e. clusters).  The main idea is to find a low-dimensional embedding and then to project data points to new space.  Recursive    bi-partitioning  method is widely used in many clustering algorithm  for the
multi-class problem.  As a basic idea, for given  graph  $G$ a   traditional spectral clustering algorithm uses the first $k$  eigenvectors
for  $\Delta_{G}$   as a low-dimensional embedding of the graph. Different graph Laplacian  and their basic properties, spectral clustering algorithms are described in \cite{Lux, MV}.

Let  $G = (V, E)$     be an undirected graph with vertex set $ V ={\{v_1, \cdots,  v_n\}}$  or simply $ V ={\{ 1, \cdots,  n\}},$ and $E$ is the set of edges.   The weighted adjacency
matrix    $W$   encode the similarity of pairwise data points, or weight $w_{ij} \ge 0$   between two vertices  $v_i$   and $ v_j$;
 \[
  W = (w_{ij} )\quad i,j=1,\cdots ,n.
  \]
  Note that   $G$   being  undirected means   $w_{ij} = w_{ji}.$  The degree of a vertex  $i \in V$  denoted by $d_i$     is
\[
d_i =\sum_{j\in V} w_{ij}.
\]
The degree matrix $D$ is defined as the diagonal matrix with
the degrees  $d_1, . . . , d_n$  on the diagonal.

  For given   graph $(V,   E)$  and  a  subset of vertex
 $C \subset V$  the   $\text{Cut}(C,C^c)$  (or the perimeter   $|\partial  C|$)  is defined by
\[
  \text{Cut}(C,C^c)  := \sum_{i \in C, j\in C^{c}} w_{ij},
\]
 where  $C^c$ is  the complement of $C$.
  The  ratio Cheeger cut   $\text{RCC}(C, C^c)$  and
normalized Cheeger cut $\text{NCC}(C, C^c)$ are defined  respectively, by
\[
\text{RCC}(C, C^c)= \frac{\text{cut}(C,C^c)}{\min(|C|, |C^c|)}, \quad \text{NCC}(C, C^c)= \frac{\text{cut}(C,C^c)}{\min(\text{vol}|C|, \text{vol}|C^c|)}.
 \]
 The minimum is achieved if  $ |C|= |C^c|.$

 For   function  $f : V\rightarrow \mathbb{R},$
the unnormalized  $p$-Laplace  operator  $\Delta_{p}^{u}$  and   normalized  $p$-Laplace operator  $\Delta_{p}^{n}$  are  defined as follow(depends  on the choice of inner product) :
\[
(\Delta_{p}^{u} f)_i=\sum_{j} \phi_{p}(f_i- f_j), \quad  \quad  (\Delta_{p}^{n} f)_i=\frac{1}{d_i} \sum_{j} \phi_{p}(f_i- f_j),
\]
where
\[
\phi_{p}(\cdot)=|\cdot|^{p-1} \rm{sign}(\cdot),
\]
 for $p=2$ we get  $\phi_{p}(x)=x.$  Also 	for  functions  $f, g  : V \rightarrow  \mathbb{R}$ the inner product is defined by
 \[
 < f,g >=\sum_{i\in V} d_i \, f_{i} g_{i}.
 \]
     Next consider the following  problem.  Find $f:V\rightarrow   \mathbb{R} $  and $\lambda \in \mathbb{R}$  such that
\[
(\Delta_{p}^{u} f)_i=  \lambda  \phi_{p}(f_i), \quad \forall  i \in V.
\]
We call $  \lambda $  an eigenvalue of  $\Delta_{p}^{u}$  associated with eigenvector  $f.$  In the case $ p = 2,$   the operator    $\Delta_{p}^{u} $  is  the regular
graph Laplacian given by
\[
\Delta_{2}^{u}= L =D-W.
\]
 The Rayleigh quotient is  defined  by
 \[
  R_p(f)=\frac{Q_{p}(f)}{\|f\|_{p}^{p}}=  \frac{\sum_{ij}(f_{i}-f_{j})^p}{2 \|f\|^p},
 \]
   with
 \[
 Q_{p}(f):=<f,\Delta_{p}^{u} f > =\frac{1}{2}  \sum_{i,j} w_{ij}|f_{i}-f_{j}|^{p}, \quad \quad  \|f\|^{p}_{p}= \sum_{i} |f_{i}|^{p}.
 \]
 Similar to continuous case, $f$  is a $p$-eigenfunction of $\Delta_{p}^{u}$  if and only if the functional $R_p$   has a critical point
at  $ f \in  \mathbb{R}^{|V|}$. The corresponding eigenvalue $\lambda_p $  is given as
\[
\lambda_p=R_{p}(f).
\]
The first eigenvalue is zero and the first eigenvector is a  constant vector. Then to find the second eigenvalue we  minimize  the Rayleigh quotient   $R_p(f)$ over all $f$ such that
\[
\sum_{i\in V} f_i=0,
\]
 which is the requirement of Algorithm 2  presented in previous section.

We point out that in \cite{BH}  is shown  that   the   cut  obtained by thresholding the second eigenvector of $p$-Laplace converges to optimal Cheeger  cut as $p$ tends to 1 and
\[
 \lambda_{2}(\Delta_1) =RCC^{*}.
 \]
However, finding optimal ratio Cheeger cut  $RCC^{*} = \underset {C\subset V} {\text{min}} \, RCC$ is NP-hard problem. To see more about minimization of $R_p(f)$ with constraint $\sum_{i\in V} f_i=0,$   see\cite{BH}.

\section{Numerical implementation}
In this section,  we briefly explain about numerical approximation of the following problem
 \begin{equation}\label{Sun1}
\left \{
\begin{array}{ll}
-\Delta_{p}  u =  f(x)   &    \textrm{ in  }      \partial  \, \Omega,\\
 u  =0      &    \textrm{ on }      \partial  \, \Omega.
  \end{array}
\right.
\end{equation}
 For $p> 1$   the solution of Problem (\ref{Sun1}) is the unique  minimizer of the following
functional
\[
E(u)= \underset{  u\neq 0}{\underset{u \in W^{1,p}_{0}(\Omega)}{ \min}}   \int_{\Omega} \frac{1}{p}  |\nabla u(x)|^{p}- f(x) u(x) \, dx.
\]
Equation in (\ref{Sun1})  is understood  in weak sense:

\begin{equation}\label{Sun2}
\int_{\Omega}|\nabla u|^{p-2} \nabla u \cdot \nabla\phi \, dx=   \int_{\Omega}f(x)  \, \phi \, dx,  \quad \forall \, \phi \in W^{1,
p}_{0}(\Omega).
\end{equation}
The discretization of problem is as follows. Let  $T_{h}$  be a regular triangulation of $\Omega_{h}$
which is composed of disjoint open regular triangles $ T_i,$  that is,
\[
\overline{\Omega}_{h}= \bigcup_{i=1}^{n} \overline{T}_i.
\]
Considering the  regularity for the solution of the $p$-Laplace equation,
  we  deal with continuous piecewise linear element.   Consider a  finite dimensional
 subspace $V_h$   of $C^0(\Omega_h),$  such
that the restriction      on elements of   $T_h,$  where $P_1$  is the linear function space:
\[
V^{h}_{0}={\{ v\in H^{1}_{0} : v \lfloor_{T} \in  \mathcal{P}_{1}, \quad \forall T \in T_h}\}.
\]
Assume  $u_n \in  V^{h}_{0} $  be the current approximation the,  we associate the  residual denoted by $R_{n}$

\[
  R_{n}   =  f(x)    -  \nabla( | \nabla u_{n} |^{p-2} \nabla u_{n}).
 \]
Equivalently,
\[
\int_{\Omega} R_{n}  \phi \, dx  =   \int_{\Omega}( f(x)\phi  -  | \nabla u_{n} |^{p-2} \nabla u_{n}  \cdot \nabla\phi )\, dx.
 \]
Then  to update $u_n$  and obtaining next approximation  denoted  by $u_{n+1}$
\[
u_{n+1}  = u_{n}  + \alpha_{n} w_{n},
\]
where $w_n$  is determined solution of linearized $p$-Laplace and the source term residual, i.e.,

 \begin{equation}\label{Sun3}
\left \{
\begin{array}{ll}
-\Delta  w_{n}  =  R_{n}    &    \textrm{ in  }       \Omega,\\
 w_{n}   =0      &    \textrm{ on }      \partial  \, \Omega,
  \end{array}
\right.
\end{equation}
Or
\[
  \int_{\Omega}(\varepsilon+ |\nabla u_{n} |)^{p-2}\,   \nabla w_{n}  \cdot \nabla\phi \, dx=   \int_{\Omega}R_{n}   \phi \, dx.
  \]

  The step length $\alpha_{n}  $  in  search direction $ w_{n}$ can be obtained as
  \[
  E(u_{n}  + \alpha_{n} w_{n} ) =\underset{ \alpha}{ \text{min  } }     E(u_n + \alpha w_n).
\]

\subsection{Numerical Examples }

 This section provides some examples of numerical approximations to the $p$-Laplace eigenvalue  problem for different values of $p$. For initial guess we use the second eigenfunction of Laplace operator.
 \begin{exam}\label{Oned}
	 Here, we verify our algorithm by invoking it in dimension one.  Let $ \lambda_{p}$ denotes the second eigenvalue in the interval $(a \,  b)$  then by result in  \cite{DM,Lin} we have
 \begin{equation}\label{Sun86}
\sqrt[p]{\lambda_{p}}=\frac{\sqrt[p]{p-1} \pi} {(b-a)  p \sin(\frac{\pi}{p})}.
\end{equation}
Let $\Omega=(-2,\, 2)$ then for $p=100$ the above formula gives:
\[
\sqrt[50]{\lambda_{100}} =2.0944 .
 \]
For      $p=50, p=100$     our approximation of second eigenvalue  are
\[
    \sqrt[50]{\lambda_{50} }=2.1717   \quad \quad    \sqrt[100]{\lambda_{100}}=2.0981,
    \]
with asymptotic given by \eqref{98} or \eqref{Sun86}
\[
\underset{p\rightarrow  \infty}{\lim}  \sqrt[p]{\lambda_{p}}=2.
\]

\begin{figure}[h!]
{\includegraphics[width=.7 \columnwidth]{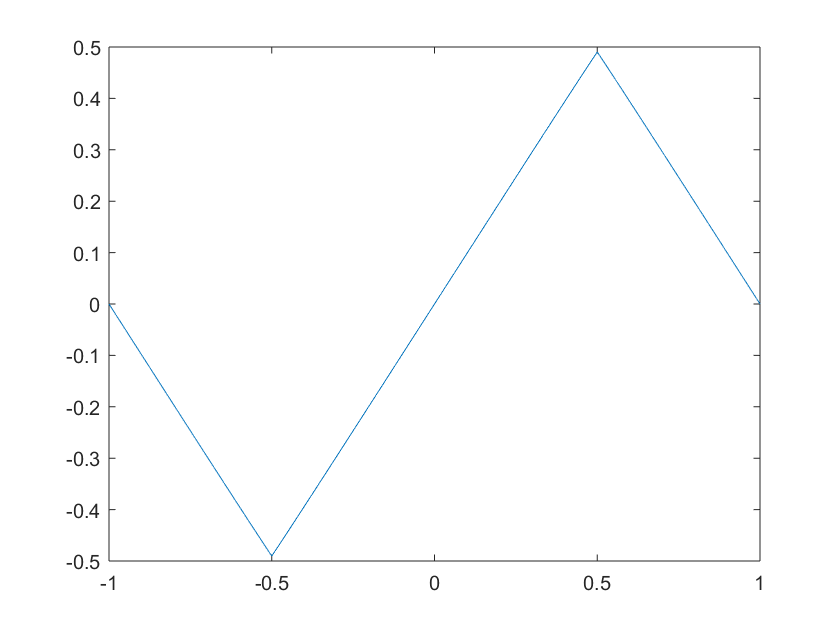}}
  \caption{ The second eigenfunction for $p=100$. }
  \label{fig:E}
\end{figure}

\end{exam}

\begin{exam}\label{E1}
 Let $p=2$ and  the domain be $ \Omega=[-2 \quad 2]\times [-2  \quad   2]$.  The second eigenvalue  is  $ \lambda_{2}=  \frac{5 \pi^2}{16}$. We set $ \triangle x =\triangle y=.005.$     The initial value is    given in Figure 2.   and  our approximate value after 50 iterations  is:  3.0843295.
 with
 \[
 | \lambda_{2}-\lambda_{2}^{(50)}|\le 0.002.
 \]

\begin{figure}[h!]
{\includegraphics[width=.7 \columnwidth]{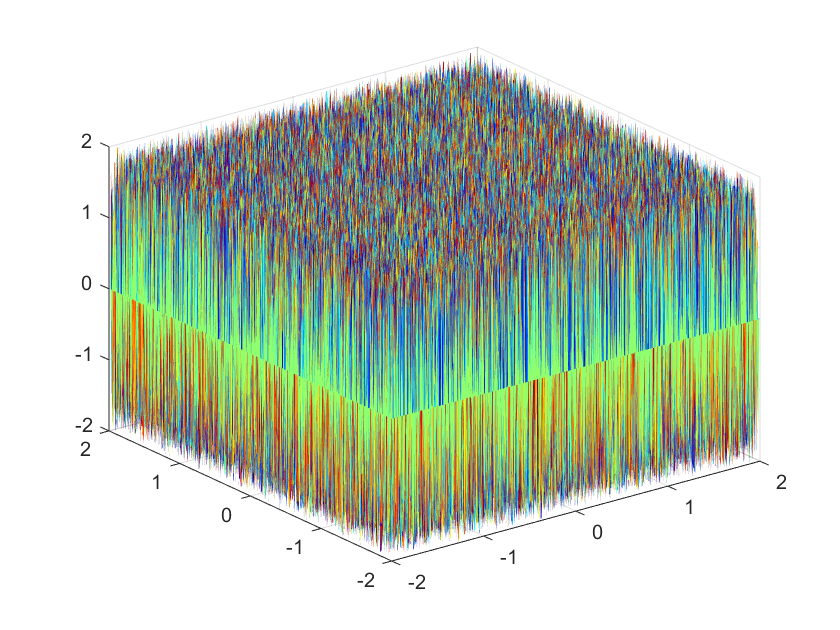}}
  \caption{ Initial guess}
  \label{fig:E}
\end{figure}

  In Figure 3, and 4   the second eigenfunctions for Dirichlet and Neumann  cases are shown.

\begin{figure}[h!]
{\includegraphics[width=.7 \columnwidth]{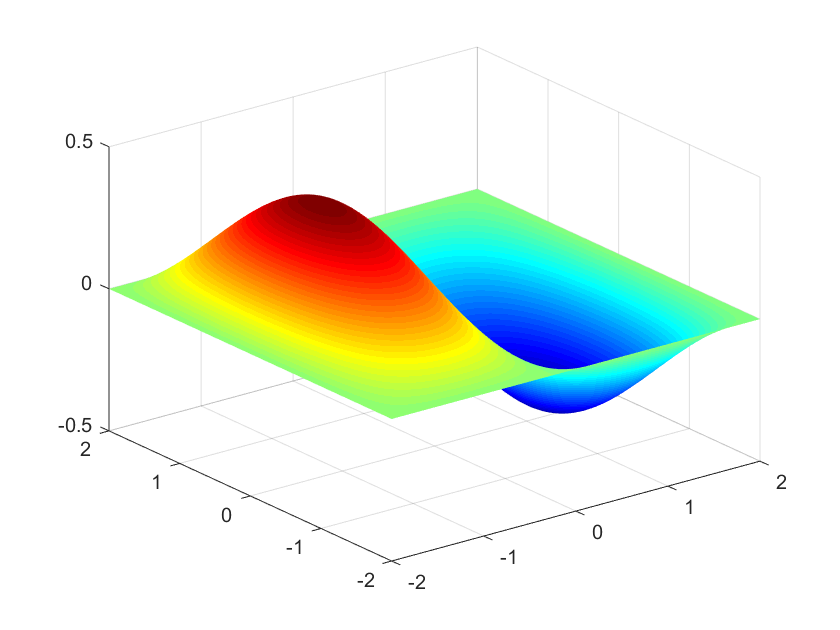}}
  \caption{Surface of $u_2$. }
  \label{fig:E}
\end{figure}

\begin{figure}[h!]
{\includegraphics[width=.7 \columnwidth]{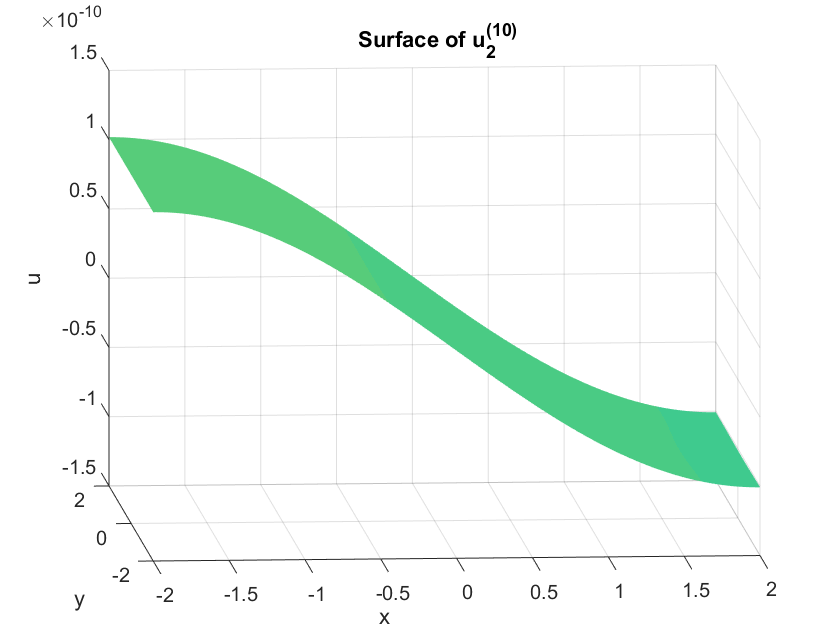}}
  \caption{Surface of $u_2$ for Neumann case.  }
  \label{fig:E}
\end{figure}

\end{exam}

\begin{exam}\label{E1}

Let domain be the  square $ [  0 \quad 2]\times  [  0 \quad 2]$. Picture \ref{fig:E} shows the second eigenfunction for $p=10.$

\begin{figure}[h!]
{\includegraphics[width=.8 \columnwidth]{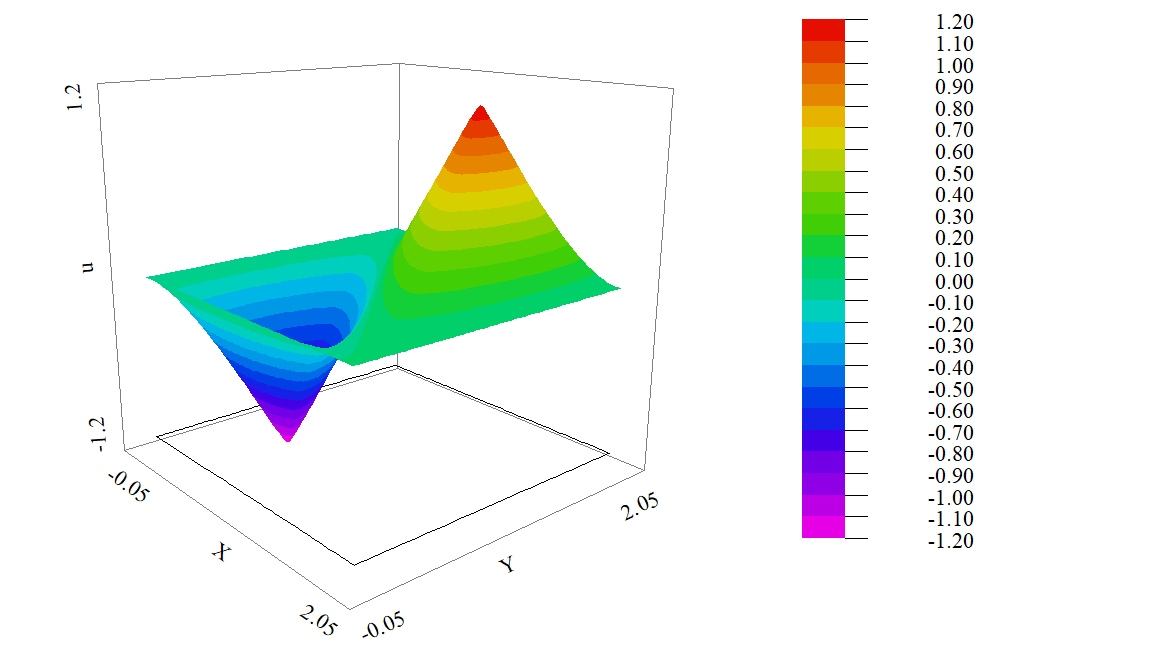}}
  \caption{Surface of $u_2$ for $p=10$.}
  \label{fig:E}
\end{figure}

\end{exam}

\begin{exam}\label{E10}

In this experiment, we use $\varepsilon$-graph.  The points are generated by    standard uniform distribution.   We choose $n=5000$, and  $\varepsilon=.05$.

 \begin{figure}[h]
          	\centering
          	\subfloat[  ]{\includegraphics[width=0.6\textwidth]{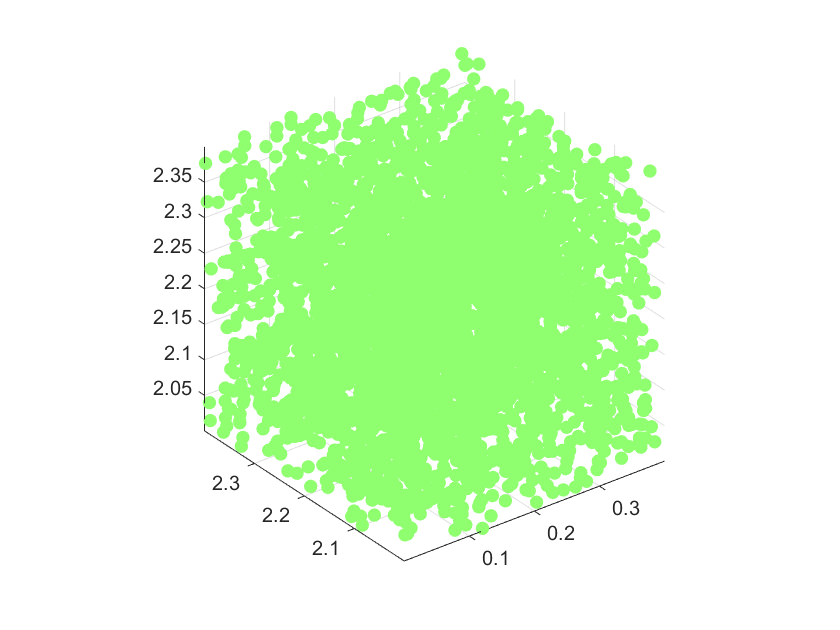}}
          	\quad
          	\subfloat[]{\includegraphics[width=0.80\textwidth]{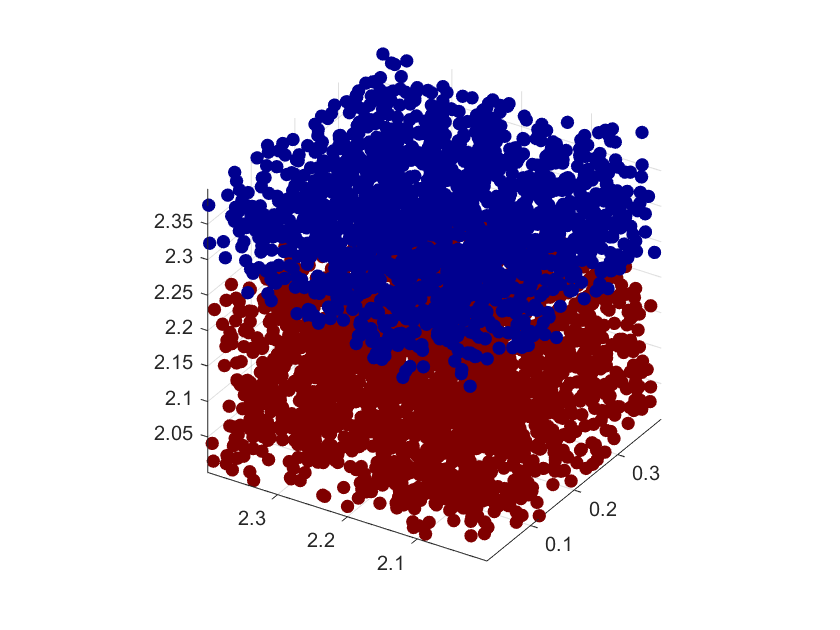}}
          	  	
          	\caption{ Top figure $n=5000$ points with   $\varepsilon=0.02.$  The second eigenvector for p=1.25}
          	\label{square-opd-fig}
          \end{figure}
\end{exam}

\section*{Acknowledgements}
{\scriptsize
	 \thanks{F. Bozorgnia was supported by the FCT post-doctoral fellowship
 SFRH/BPD/33962/2009 and by Marie Skłodowska-Curie grant agreement No. 777826 (NoMADS). The author is thankful to  Leon Bungert  for insightful discussion contributing to deepen on subsection 3.1}.}

\renewcommand{\refname}{REFERENCES }

\end{document}